\documentclass[11pt,reqno]{amsart}
\usepackage{amsmath,amssymb,amscd,amsxtra,amsfonts,mathrsfs}
\usepackage{bm,epsf,graphicx,epsfig,color,latexsym,cite,cases}
\usepackage[normalem]{ulem}
\usepackage{subcaption}
\usepackage{float}
\usepackage[colorlinks,linktocpage,linkcolor=blue]{hyperref}

\allowdisplaybreaks[4]

\newtheorem{theorem}{Theorem}[section]
\newtheorem{corollary}[theorem]{Corollary}
\newtheorem{lemma}[theorem]{Lemma}

\newtheorem{remark}[theorem]{Remark}

\newtheorem{definition}[theorem]{Definition}

\numberwithin{equation}{section}

\newcommand{\hL}{\mathcal{L}}

\def\n{\partial{\overrightarrow{\bf n}}}

\begin{document}

\title[Inverse Random Source Problems for Stochastic Heat and Wave Equations]{Uniqueness of Inverse Random Source Problems for Stochastic Heat and Wave Equations}

\author{Xu Wang}
\address{State Key Laboratory of Mathematical Sciences, Academy of Mathematics and Systems Science, Chinese Academy of Sciences, Beijing 100190, China, and School of Mathematical Sciences, University of Chinese Academy of Sciences, Beijing 100049, China.}
\email{wangxu@lsec.cc.ac.cn}

\author{Guanlin Yang}
\address{Corresponding author. State Key Laboratory of Mathematical Sciences, Academy of Mathematics and Systems Science, Chinese Academy of Sciences, Beijing 100190, China, and School of Mathematical Sciences, University of Chinese Academy of Sciences, Beijing 100049, China.}
\email{yanguanlin@lsec.cc.ac.cn}

\author{Zhidong Zhang}
\address{School of Mathematics (Zhuhai), Sun Yat-sen University, Zhuhai 519082, Guangdong, China, and Guangdong Province Key Laboratory of Computational Science, Sun Yat-sen University, Guangzhou 510000, Guangdong, China.}
\email{zhangzhidong@mail.sysu.edu.cn}

\thanks{Xu Wang is supported by the National Key R\&D Program of China (2024YFA1015900 and 2024YFA1012300), the NSF
of China (12288201), and the CAS Project for Young Scientists in Basic Research (YSBR-087). Zhidong Zhang is supported by the National Key R\&D Plan of China (2023YFB3002400).}



\begin{abstract}
This paper investigates an inverse random source problem for stochastic evolution equations, including stochastic heat and wave equations, with the unknown source modeled as \( g(x)f(t)\dot{W}(t) \). The research commences with the establishment of the well-posedness of the corresponding stochastic direct problem. Under suitable regularity conditions, the existence of stochastic strong solutions for both the stochastic heat and wave equations is demonstrated. For the inverse problem, the objective is to uniquely recover the strength \(|f(t)|\) of the time-dependent component of the source from the boundary flux on a nonempty open subset. The uniqueness of the recovery for both the stochastic heat and wave equations is proven, and several numerical examples are given to verify the theoretical results.
\end{abstract}

\maketitle

\section{Introduction}

\subsection{Problem and background}
We study in this paper an inverse random source problem for both heat and wave equations driven by a Gaussian random noise, which can be summarized as the following stochastic partial differential equation (SPDE) form:
\begin{equation} \label{eq:model}
\left\{
\begin{aligned}
\partial_t^m u(x, t) -\Delta u(x, t) &= g(x)f(t)\dot{W}(t), && (x, t) \in D \times (0, T], \\
u(x, t) &= 0, && (x, t) \in \partial D \times [0, T],
\\
u(x, 0) &= 0, && x \in D,\\
\partial_t u(x, 0) &= 0, && x \in D,\quad \text{(only if $m=2$)}.
\end{aligned}
\right.
\end{equation}
Here, the parameter \(m \in\{1, 2\}\), the spatial domain \(D \subset \mathbb{R}^d\) is assumed to be bounded with a Lipschitz boundary, and the source term is given in a separable form, where \(f : [0,T] \to \mathbb{R}\) and \(g : D \to \mathbb{R}\) are  deterministic functions, and $\dot{W}$ is the time-dependent white noise defined on a filtered probability space $(\Omega, \mathcal{F}, \{\mathcal{F}_t\}_{t\geq 0}, \mathbb{P})$. In the source term, the function $f$ is unknown, while the spatial component $g$ is controllable. We use the boundary measurements and denote the observation region by $D_{o b} \subset \partial D$. The mathematical formula for the flux data is given by 
$\frac{\partial u}{\partial \overrightarrow{\mathbf{n}}}|_{D_{o b} \times(0, T)}.$
So the inverse random source problem in this work can be described as follows:
\begin{equation}\label{inverse_problem}
\text { use }\left.\frac{\partial u}{\partial \overrightarrow{\mathbf{n}}}\right|_{D_{ob} \times(0, T)} \text { to uniquely recover the strength } |f| \text{ of the unknown } f \text {.}
\end{equation}

The research on inverse source problems in the deterministic setting has a long history and many representative works are generated. We refer to the books \cite{JR84, I90, POV00} for a comprehensive survey. Denote the unknown source by $F(x,t)$. To establish the well-posedness of the inverse source problem in the sense of Hadamard, a priori information concerning the source $F(x,t)$ is necessary. Based on different types of measurements, most of existing works can be briefly classified into recovering space-dependent or time-dependent part of the source from the space-dependent or time-dependent data (cf. \cite{JR68,CE76,R80,I91,HR01,PT03}). Furthermore, \cite{RZ20,LZZ22} considered the case $F(x,t)=g(x)f(t)$ when $g(x)$ and $f(t)$ are both unknown. It is worth mentioning that to ensure the uniqueness of the inverse problem, some prior regularity assumptions on $g$ and $f$ are needed, usually within a proper subspace of $L^2(D)$ and $L^2([0,T])$. 

However, the research on the recovery of sources with uncertainties, which are known as inverse random source problems, seems relatively scarce. The challenges lie in the roughness of the random noise. Specifically, the Gaussian white noise $\dot{W}(t)$ considered in the present paper belongs to the Sobolev space $W^{-\frac{1+\epsilon}{2},p}([0,T])$ of negative order with $\epsilon>0$ being arbitrarily small and $p>1$, which results in low regularity of the solution. The roughness leads to the failure of the methods that work for the deterministic case, and new framework should be explored for inverse random source problems. 

For time-harmonic equations, some literature has focused on the numerical computation, uniqueness theory, and stability theory of inverse random source problems. 
For example, \cite{BCL16,BCL17} considered random sources driven by Gaussian white noise in the form $g(x)\dot{W}(x)$, and studied the numerical computation of the strength function $g^2$.
Authors in \cite{LW2,LW3,LLM21} investigated the uniqueness for the recovery of the strength of generalized Gaussian random sources based on the microlocal analysis and theory of pseudo-differential operators. The stability of the recovery was studied in \cite{LL24} for a random source driven by Gaussian white noise, and in \cite{WXZ24} for generalized Gaussian random sources.

For time-dependent problems, the existence of both space and time variables makes the framework for inverse random source problems quite different from the time-harmonic case.
Most studies focus on the case that the random source is space-time separable, and recover either the temporal component $f(t)$ or the spatial component $g(x)$ involved in the random source. 
For the recovery of the spatial component $g(x)$, we refer to \cite{NHZ20,FLW20,FZLW22,LLZ23} for the uniqueness result, together with the recovery formula, under conditions that $g\in L^2(D)$ and the randomness lies only in the temporal direction such that the spectral expansion or the unique continuation principle can be utilized.
For the recovery of the temporal component $f(t)$, \cite{GLWX21,FYLW24} employed the Fourier transform to uniquely determine the phaseless Fourier modes $|\hat f(\xi)|$ under conditions that $f\in L^2([0,T])$ and the randomness lies in the spatial direction. Then the PhaseLift method was applied to solve the phase retrieval problem and derive numerically the strength $|f(t)|$.
We also refer to \cite{LWZ20,LZ15} for stability estimates of inverse random source problems driven by time-dependent random noise.


In this article, we focus on the unique recovery of the temporal component $f(t)$ in the random source $g(x)f(t)\dot W(t)$, and aim to propose a unified approach, together with a computable recovery formula, for both heat and wave equations. 
Since the unknown temporal function $f$ is coupled with the noise $\dot{W}(t)$, it is hard to separate the target function from the stochastic integral, as achieved by the Fourier transform method or the spectral expansion approach in the existing results.

\subsection{Main results and outline}
For the direct problem related to SPDE  \eqref{eq:model}, we focus on the existence and uniqueness of the stochastic strong solution other than stochastic mild or weak solutions, primarily motivated by the analytical framework required for the inverse problem. More precisely, to establish the uniquely recovery result, an auxiliary function framework, together with Green's formula, will be utilized, which requires a higher spatial regularity of the solution $u$ (see Theorem \ref{thm:strong}).

Denote by $\dot{H}^s(D)=Dom((-\Delta)^{\frac s2})$ with $s\ge0$, where the precise characterization of $\dot{H}^s(D)$ using spectral expansion and its relation to classical Sobolev spaces will be given in Section 2. The well-posedness theorem for the direct problem of \eqref{eq:model} is stated as follows.

\begin{theorem}\label{thm:direct}
Assume that $f\in L^2([0,T])$ and $g\in \dot{H}^1(D)$. Then \eqref{eq:model} admits a unique strong solution for $m\in \{1,2\}$. Moreover, the strong solution $u\in L^2([0,T];\dot{H}^2(D))$  for $\mathbb P$-a.s. $\omega\in\Omega$.
\end{theorem}

It is also worth noting that Theorem \ref{thm:direct} indicates that the solutions to the stochastic heat and wave equations exhibit the same regularity under identical source terms. This is quite different from the deterministic case, where parabolic equations exhibit a stronger regularizing effect than hyperbolic equations, typically leading to higher regularity of solutions even with the same source term. The reason is that the regularity of solutions of SPDE \eqref{eq:model} is determined by the regularity of the stochastic integral term, and hence is limited by the regularity of the random source. 

For the inverse random source problem \ref{inverse_problem}, we give the uniqueness theorem as follows.

\begin{theorem}\label{thm:inverse}
Under the assumptions of Theorem~\ref{thm:direct}, assume further that $g \in \dot{H}^{d-1+\epsilon}(D)$ for some arbitrary $\epsilon > 0$, where $d \in \{1,2,3\}$ denotes the spatial dimension. Then $|f|$ can be uniquely determined by measurements \small{$\frac{\partial u}{\partial \overrightarrow{\mathbf{n}}}\Big|_{D_{o b} \times(0, T)}$} for $m\in\{1,2\}$, where $D_{o b} \subseteq \partial D$ is a nonempty open subset of $\partial D$.
\end{theorem}

Theorem \ref{thm:inverse} confirms that the data from any nonempty open subset of the boundary is sufficient to support the uniqueness of the inverse problem \eqref{inverse_problem}, which could contribute to cost reduction in practical applications of inverse problems. 

The rest part of the paper is outlined as follows. In Section 2, we provide several preliminary results and give the different concepts of solutions for SPDEs in the abstract evolution form. In Section 3, we show the well-posedness of \eqref{eq:model} in the strong sense for both stochastic heat and wave equations, together with properties for strong solutions (Theorem \ref{thm:strong}) that will be used in the study of the inverse problem. Section 4 is devoted to the proof of the uniqueness result for the inverse random source problem stated in Theorem \ref{thm:inverse}, which is derived from a representation formula based on the data (see Lemma~\ref{lem:representaition}). With the uniqueness theorem, several numerical examples are presented in Section 5 to validate the effectiveness of the reconstruction. In Section 6, some concluding remarks together with some directions for future work are given for inverse random source problems.

\section{Preliminaries}
Let $(H,\langle \cdot,\cdot \rangle)$ be a separable Hilbert space endowed with the norm $\|\cdot\|$ induced by the inner product $\langle\cdot,\cdot\rangle$, and $\{W(t)\}_{t\ge0}$ be a real-valued standard Brownian motion defined on a filtered probability space $(\Omega, \mathcal{F}, \{\mathcal{F}_t\}_{t\geq 0}, \mathbb{P})$. Hereinafter, we omit the sample path $\omega\in\Omega$ in a stochastic process and use the notation $X(t)=X(t,\omega)$ for simplicity. In this section, we consider a stochastic process $\{X(t)\}_{t\in[0,T]}$ taking values in $H$, satisfying the following abstract stochastic evolution equation
\begin{equation}\label{eq:abstract}
\left\{
\begin{aligned}
dX(t) & = -AX(t)\,dt + B(t)\,dW(t), \quad t\in(0,T),\\
X(0) & = 0,
\end{aligned}
\right.
\end{equation}
where $A: \mathcal{D}(A) \subset H \rightarrow H$ is a  densely defined, self-adjoint, and positive-definite linear operator with a compact inverse and $B: (0,T) \rightarrow H$ is a deterministic $H$-valued function with $\int_0^t \|B(s)\|^2ds<\infty$. An example of the operator $A$ is the negative Laplacian $-\Delta$ defined on a smooth bounded domain $D \subset \mathbb{R}^d$ with homogeneous Dirichlet boundary conditions.

In this section, we introduce several notations related to the operator $A$, and then present the concepts of mild and strong solutions to stochastic evolution problem \eqref{eq:abstract}. For the  general solution theory to SPDEs driven by space time white noise $\dot{W}$, we refer the readers to \cite{DZ14}.

\subsection{Eigensystem of $A$ and the semigroup $\{S(t)\}_{t\geq0}$}
In this part, we introduce the eigensystem of the operator $A$ and the associated semigroup $\{S(t)\}_{t\geq0}$ generated by $A$. Since $A$ is self-adjoint and positive-definite with a compact inverse, it admits an eigensystem $\{\lambda_n, \phi_n\}_{n=1}^\infty$ satisfying
\[
A \phi_n = \lambda_n \phi_n, \quad 0 < \lambda_1 \leq \lambda_2 \leq \cdots \to \infty,
\]
where $\{\phi_n\}_{n=1}^\infty$ forms an orthonormal basis of $H$.

We define the scale of Hilbert spaces associated with powers of $A$ as follows. For any $s \geq 0$, define the norm
$$
\|v\|_{\dot{H}^s}:=\left\|A^{\frac s2} v\right\|=\left(\sum_{j=1}^\infty \lambda_j^s\left\langle v, \phi_j\right\rangle\right)^{\frac12}, 
$$
and  the corresponding spaces
$$
\dot{H}^s:=Dom(A^{\frac s2})=\left\{v \in H:\|v\|_{\dot{H}^s}<\infty\right\}.
$$

Moreover, the operator $A$ generates a strongly continuous semigroup $\{S(t):=e^{-tA}\}_{t\geq0}$ on $H$, which admits the following spectral representation:
\[
S(t)v = \sum_{n=1}^\infty e^{-\lambda_n t} \langle v, \phi_n \rangle \phi_n, \quad v \in H, \ t \geq 0.
\]
The semigroup $\{S(t)\}_{t\geq0}$ plays an important role in the formulation of solutions to SPDEs.

\subsection{Stochastic mild and strong solutions}

In the theory of SPDEs, it is essential to distinguish between different notions of solutions, depending on the interpretation of the equation and the regularity of coefficients. In this part, we provide the definitions of mild and strong solutions to \eqref{eq:abstract}, which will play a central role in our subsequent analysis.

\begin{definition}[Stochastic mild solution]
An $H$-valued stochastic process $\{X(t)\}_{t\in[0,T]}$ is called a mild solution of  \eqref{eq:abstract} if 
\begin{equation}\label{solution:mild}
X(t)=\int_0^t S(t-s)B(s)dW(s) \quad \mathbb P\text{-a.s.}
\end{equation}
for $t\in[0,T]$ with the stochastic integral on the right-hand side being well-defined.
\end{definition}
\begin{remark} \label{rem:mild}
Assume that $B\in L^2([0,T];H)$, i.e., $\int_0^t \|B(s)\|^2ds<\infty$, the stochastic integral in \eqref{solution:mild} is well-defined, since
\[
\int_0^t\|S(t-s)B(s)\|^2ds \leq \int_0^t \|S(t-s)\|^2_{\hL (H)} \|B(s)\|^2ds <\infty,
\]
where the uniform boundedness property of semigroup $\|S(t-s)\|_{\hL (H)}\leq C$ (cf.\cite[Lemma B.9 (1)]{K14}) is utilized.
\end{remark}

In our later analysis, we would require stronger spatial regularities of the solution, which is fulfilled in the sense of stochastic strong solution.

\begin{definition}[Stochastic strong solution]
A $Dom(A)$-valued stochastic process $\{X(t)\}_{t\in[0,T]}$ is called a strong solution of  \eqref{eq:abstract} if 
\begin{equation}\label{solution:strong}
X(t)=-\int_{0}^{t}A X(s)ds+\int_0^t B(s)dW(s) \quad \mathbb P\text{-a.s.}
\end{equation}
for  $t\in[0,T]$ with integrals on the right-hand side being well-defined, where $Dom(A)$ denotes the domain of the operator $A$.
\end{definition}

For simplicity, the notation `$\mathbb P$-a.s.' is omitted in the following sections if there is no confusion.

It is well known that a strong solution to \eqref{eq:abstract} is also a mild solution (cf. \cite[Section 5.1]{DZ14}). Conversely, we will demonstrate in the next section that the stochastic mild solution becomes a strong solution under additional assumptions on the coefficients.

\section{Well-posedness of the direct problem}

In this section, we study the existence and uniqueness of the strong solution to \eqref{eq:model} and present properties of the strong solution that will be used in the subsequent analysis for the inverse problem. To begin with, we rewrite \eqref{eq:model} into the framework of the abstract form \eqref{eq:abstract}, where the cases of the heat equation (\(m=1\)) and the wave equation (\(m=2\)) are treated separately. Theorem \ref{thm:direct} is then established, providing existence and regularity results under suitable assumptions on the source term. Finally, Theorem \ref{thm:strong} presents an integration by parts formula satisfied by the strong solution.

Throughout the remainder of the paper, we set \(H = L^2(D)\), equipped with the inner product $\langle u,v\rangle=\int_Du(x)v(x)dx$ and the induced norm \(\|\cdot\|=\|\cdot\|_{L^2(D)}\), and set \(A = -\Delta\) with homogeneous boundary conditions on $\partial D$. Then it holds $\dot{H}^2(D)=H^2(D)\cap H^1_0(D)$ and $\dot{H}^1(D)= H^1_0(D)$, where $H^1_0(D)$ and $H^2(D)$ denotes the standard Sobolev space. Moreover, for any $f\in\dot{H}^1(D)$, one has
$\|f\|_{\dot{H}^1(D)} = \|\nabla f\|$ (cf. \cite[Theorem 6.4]{LV03}),
which together with Poincar\'e's inequality implies the equivalence of the norms $\|\cdot\|_{\dot{H}^1(D)}$ and $\|\cdot\|_{H^1(D)}$.

\subsection{Stochastic heat equation}

We first consider the parabolic case with \(m=1\) in \eqref{eq:model} and study the corresponding direct problem.  Setting $B(x,t)=g(x)f(t)$, then \eqref{eq:model} can be rewritten into the following abstract form
\begin{equation}\label{eq:parabolic}
\left\{\begin{aligned}
du(\cdot,t) & = -Au(\cdot,t)dt+B(\cdot,t)dW(t), \quad t>0,\\
u(\cdot,0) & =0.
\end{aligned}\right.
\end{equation}

Given \( f \in L^2([0, T]) \) and \( g \in H \), it follows that \( B \in L^2([0, T]; H) \). Hence \eqref{eq:parabolic} admits an $H$-valued mild solution according to Remark \ref{rem:mild}. Since the strong solution takes values in \( Dom(A) = \dot{H}^2(D) \), higher regularity is required for the random source term. In what follows, we show the well-posedness of the strong solution under the assumption \( g \in \dot{H}^1(D) \), as stated in Theorem \ref{thm:direct}.

\begin{proof}[Proof of Theorem \ref{thm:direct} ($m=1$)]
Under the assumptions of $f\in L^2(D)$ and $g\in \dot{H}^1(D)$, we will prove that the mild solution is also a strong solution. 

We first show that the mild solution $u$ satisfies $u\in L^2(\Omega;L^2([0,T];\dot{H}^2(D)))$. By the smoothing effect of the analytic semigroup $S(t)$ (cf.\cite[Chapter 2, Theorem 6.13]{P83}), it holds for any $t>0$ that 
\[
S(t)g(\cdot) \in Dom(A).
\]
Interchanging of stochastic integral with the Laplacian operator (cf.\cite[Proposition 4.30]{DZ14}) and It\^{o}'s isometry yield that
\begin{align}
\mathbb{E}\left[\|u(\cdot,t)\|^2_{\dot{H}^2(D)}\right] &= \mathbb{E}\left[\left\|A\int_0^t S(t-s)f(s)g(\cdot)dW(s)\right\|^2\right] \notag\\
&= \mathbb{E}\left[\left\|\int_0^t AS(t-s)f(s)g(\cdot)dW(s)\right\|^2\right]\notag\\
&=\int_0^t \left\|AS(t-s)g(\cdot)\right\|^2f^2(s)ds \notag\\
&=\sum_{j=1}^{\infty}\lambda_j^2\langle g, \phi_j\rangle^2 \int_0^t e^{-2\lambda_j(t-s)}f^2(s)ds. \label{eq:2-1}
\end{align}
By Young's convolution inequality, we obtain that 
\[
 \int_0^t e^{-2\lambda_j(t-s)}f^2(s)ds \leq \int_0^t e^{-2\lambda_js} ds \|f\|^2 = \frac{1}{2\lambda_j}(1-e^{-2\lambda_jt}) \|f\|^2< \frac{1}{2\lambda_j} \|f\|^2.
\]
It follows from \eqref{eq:2-1} that
\begin{equation*}
\mathbb{E}\left[\int_0^T\|u(\cdot,t)\|^2_{\dot{H}^2(D)}dt\right] < \int_0^T\frac12  \|f\|^2\sum_{j=1}^{\infty}\lambda_j\langle g, \phi_j\rangle^2dt= \frac{T}{2}\|f\|^2\|g\|_{\dot{H}^1(D)}^2<\infty,
\end{equation*}
which verifies $u\in L^2(\Omega;L^2([0,T];\dot{H}^2(D)))$.

Next, we prove that the mild solution $u$ satisfies \eqref{solution:strong} in the definition of stochastic strong solution. By virtue of the stochastic Fubini theorem (cf.\cite[Theorem 4.18]{DZ14}) and the dominated convergence theorem, one obtains that

\begin{align*}
\int_{0}^{t}A u(\cdot,s)ds &= \int_0^t\int_0^sAS(s-\tau) g(\cdot)f(\tau)dW(\tau)ds = \int_0^t\int_\tau^tAS(s-\tau) g(\cdot)f(\tau)dsdW(\tau)\\
&= \int_0^t\int_\tau^t\sum_{j=0}^{\infty}e^{-\lambda_j(s-\tau)}\lambda_j\langle g,\phi_j\rangle\phi_j(\cdot)f(\tau)dsdW(\tau)\\
&= \int_0^t\sum_{j=0}^{\infty}\left[ \int_\tau^te^{-\lambda_j(s-\tau)}\lambda_jds\right]\langle g,\phi_j\rangle \phi_j(\cdot)f(\tau)dW(\tau)\\
&= \int_0^t\sum_{j=0}^{\infty} (1-e^{-\lambda_j(t-\tau)})\langle g,\phi_j\rangle \phi_j(\cdot) f(\tau)dW(\tau)\\
&=g(\cdot)\int_0^t f(\tau)dW(\tau)-\int_0^tS(t-\tau)g(\cdot)f(\tau)dW(\tau)\\
&=\int_0^t B(\cdot,\tau)dW(\tau)-u(\cdot,t),
\end{align*}
which coincides with \eqref{solution:strong}.

For the uniqueness, assume that $u_1$ and $u_2$ are both strong solutions of \eqref{eq:model}. 
Let $v = u_1 - u_2 \in L^2([0,T]; \dot{H}^2(D))$. It then holds
\begin{equation*}
v(x,t) = \int_0^t \Delta v(x,s)\, ds, \qquad (x,t)\in D\times(0,T]
\end{equation*}
with homogeneous initial and boundary conditions, which satisfies
\[
\|v(\cdot,t)\|^2
= \int_0^t \langle \Delta v(\cdot,s), v(\cdot,t)\rangle\, ds
= - \int_0^t \langle \nabla v(\cdot,s), \nabla v(\cdot,t)\rangle\, ds.
\]
Integrating over $t\in[0,T]$ and applying Fubini's theorem yield
\begin{align*}
\int_0^T \|v(\cdot,t)\|^2\, dt
&= -\int_0^T \int_0^t \langle \nabla v(\cdot,s), \nabla v(\cdot,t)\rangle\, ds\, dt \\
&= -\frac12 \int_0^T \int_0^T \langle \nabla v(\cdot,s), \nabla v(\cdot,t)\rangle\, ds\, dt \\
&= -\frac12 \left\|\int_0^T \nabla v(\cdot,t)\, dt \right\|^2.
\end{align*}
 Hence $v\equiv0$, which completes the proof.
\end{proof}

\subsection{Stochastic wave equation}

For the hyperbolic case with \(m=2\) in \eqref{eq:model}, we start with deriving the abstract form of \eqref{eq:model}.

Denoting $X(t):=[u(\cdot,t),\partial_tu(\cdot,t)]^\top$, $B(x,t):=g(x)f(t)$, and $\tilde{H}:=\dot{H}^1(D)\times L^2(D)$, it holds from \eqref{eq:model} that 
\begin{align*}
dX(t)&=
\begin{bmatrix}
du(\cdot,t) \\ d\partial_tu(\cdot,t)
\end{bmatrix}=
\begin{bmatrix}
\partial_tu(\cdot,t) dt \\
\Delta u(\cdot,t)dt+B(\cdot,t)dW(t)
\end{bmatrix} \\
&= 
\begin{bmatrix}
0 & I\\
\Delta & 0
\end{bmatrix}
\begin{bmatrix}
u(\cdot,t) \\
\partial_tu(\cdot,t)
\end{bmatrix}dt+
\begin{bmatrix}
0\\B(\cdot,t)
\end{bmatrix}dW(t)\\
&=
-\tilde{A}X(t)dt+\tilde{B}(t)dW(t),
\end{align*}
where 
\[
\tilde{A}:=\begin{bmatrix}
0 & -I\\
-\Delta & 0
\end{bmatrix}, \quad
\tilde{B}(t):=\begin{bmatrix}
0\\B(\cdot,t)
\end{bmatrix}.
\]
Setting
\[
Dom(\tilde{A}):=\left\{X=(X_1,X_2)^\top: \begin{bmatrix}X_2 \\ 
				\Delta X_1 \end{bmatrix}
\in \dot{H}^1(D)\times L^2(D)
\right\}= \dot{H}^{2}(D) \times \dot{H}^1(D),
\]
then the hyperbolic case with $m=2$ in \eqref{eq:model} can be rewritten as
\begin{equation}\label{eq:hyperbolic}
\left\{\begin{aligned}
dX(t) & = -\tilde{A}X(t)dt+\tilde{B}(t)dW(t), \quad t>0,\\
X(0) & =0.
\end{aligned}\right.
\end{equation}

It can be verified that $-\tilde{A}$ also generates a $C_0$-semigroup $\{\tilde{S}(t):=e^{-\tilde{A}t}\}_{t\geq0}$ on $\tilde{H}$. To see what $\tilde{S}(t)$ is, one can use cosine and sine operator functions, and the semigroup can be formally expressed as 
$$
\tilde{S}(t)=e^{-t \tilde{A}}=\left[\begin{array}{cc}
\cos \left(t (-\Delta)^{\frac12}\right) & (-\Delta)^{-\frac12} \sin \left(t (-\Delta)^{\frac12}\right) \\
-(-\Delta)^{\frac12} \sin \left(t (-\Delta)^{\frac12}\right) & \cos \left(t (-\Delta)^{\frac12}\right)
\end{array}\right].
$$
Hence the mild solution of \eqref{eq:hyperbolic} admits the form

$$
X(t)= \int_0^t \tilde{S}(t-s) \tilde{B}(s) dW(s) =\left[\begin{array}{c}
\int_0^t (-\Delta)^{-\frac12} \sin \left((t-s) (-\Delta)^{\frac12}\right)B(\cdot,s)dW(s) \\
\int_0^t \cos \left((t-s) (-\Delta)^{\frac12}\right)B(\cdot,s) dW(s)
\end{array}\right].
$$

Based on the above form of the mild solution to \eqref{eq:hyperbolic}, we next give the proof of Theorem \ref{thm:direct} for the hyperbolic case.

\begin{proof}[Proof of Theorem \ref{thm:direct} ($m=2$)]
We will prove that the mild solution is also a strong solution. We first show that $X\in L^2(\Omega;L^2([0,T]; Dom(\tilde{A})))$, which yields $u\in L^2(\Omega;L^2([0,T];\dot{H}^2(D)))$. It suffices to verify
\begin{align*}
\mathbb{E}\left[\int_0^T\left\|\int_0^t (-\Delta)^{-\frac12} \sin \left((t-s) (-\Delta)^{\frac12}\right)B(\cdot,s)dW(s)\right\|_{\dot{H}^2(D)}^2dt\right]
 <\infty,\\
\mathbb{E}\left[ \int_0^T\left\|\int_0^t \cos \left((t-s) (-\Delta)^{\frac12}\right)B(\cdot,s)dW(s)\right\|^2_{ \dot{H}^1(D)}dt \right]<\infty.
\end{align*}
By It\^{o}'s isometry and interchange of stochastic integral with closed operators \cite[Proposition 4.30]{DZ14}, it suffices to show
\begin{align}
\int_0^T\int_0^t \| (-\Delta)^{\frac12}\sin \left((t-s) (-\Delta)^{\frac12}\right)B(\cdot,s)\|^2dsdt &< \infty, \label{eq:1-1}\\
\int_0^T\int_0^t \|(-\Delta)^{\frac12} \cos \left((t-s) (-\Delta)^{\frac12}\right)B(\cdot,s)\|^2 dsdt &<\infty.\label{eq:1-2}
\end{align}
By Parseval's identity, we obtain that 
\begin{align*}
\int_0^t \| (-\Delta)^{\frac12}\sin \left((t-s) (-\Delta)^{\frac12}\right)B(\cdot,s)\|^2ds
=& \int_0^t\|\sum_{j=0}^\infty\lambda_j^{\frac{1}{2}}\sin \left((t-s) \lambda_j^{\frac12}\right) f(s)\langle g, \phi_j\rangle \phi_j\|^2ds \\
=& \int_0^t \sum_{j=1}^\infty \lambda_j \sin^2 \left((t-s) \lambda_j^{\frac12}\right) f^2(s) \langle g, \phi_j\rangle^2 ds\\
\leq&  \sum_{j=1}^\infty \lambda_j \langle g, \phi_j\rangle^2  \int_0^t f^2(s)ds\le\|f\|^2 \|g\|_{\dot{H}^1(D)}^2<\infty,
\end{align*}
which indicates \eqref{eq:1-1}.
The analysis of \eqref{eq:1-2} follows similarly and is omitted here. 

We then show that the mild solution $X$ satisfies \eqref{solution:strong} in the definition of stochastic strong solution. Based on the representations of semigroup $\tilde{S}(t)$ and mild solution $X(t)$, the stochastic Fubini theorem yields
\begin{align*}
\int_{0}^{t}\tilde{A} X(s)ds &= \int_0^t\int_0^s\tilde{A}\tilde{S}(s-\tau) \tilde{B}(\tau)dW(\tau)ds = \int_0^t\int_\tau^t\tilde{A}\tilde{S}(s-\tau) \tilde{B}(\tau)dsdW(\tau)\\
&= \int_0^t\int_\tau^t\left[\begin{array}{c}
 -f(\tau) \cos \left((s-\tau)(-\Delta)^{\frac12}\right) g(\cdot) \\
f(\tau) (-\Delta)^{\frac12} \sin \left((s-\tau) (-\Delta)^{\frac12}\right)g(\cdot)
\end{array}\right]dsdW(\tau)\\
&= \int_0^t\left[\begin{array}{c}
 -f(\tau) (-\Delta)^{-\frac12} \sin \left((t-\tau)(-\Delta)^{\frac12}\right) g(\cdot) \\
\left[-f(\tau) \cos \left((t-\tau)(-\Delta)^{\frac12}\right) +f(\tau)\right]g(\cdot)
\end{array}\right]dW(\tau)\\
&=-X(t)+\int_0^t \tilde{B}(\tau)dW(\tau),
\end{align*}
which coincides with \eqref{solution:strong}. The uniqueness follows similarly as in the heat equation case and is omitted here, thereby finishing the proof.
\end{proof}

\subsection{Property of stochastic strong solutions}

In this part, we give the property of stochastic strong solutions to \eqref{eq:model}, which indicates that the stochastic strong solution also satisfies a modified form of the stochastic weak solution. This property is essential for obtaining the uniqueness of the inverse problem in the subsequent section.

\begin{theorem}\label{thm:strong}
Let $\{u(\cdot,t)\}_{t\in[0,T]}$ be the stochastic strong solution of \eqref{eq:model} with $f\in L^2([0,T])$ and $g\in \dot{H}^1(D)$. 
\item[(i)] If $m=1$, for any $\zeta\in L^2([0,T];H^1_0(D))$ with $\partial_t \zeta \in L^2([0,T];H^{-1}(D))$, it holds for almost every $t\in[0,T]$ that
\begin{align}\label{prop:para_strong}
\langle u(\cdot,t), \zeta (\cdot,t)\rangle=&\int_0^t\langle\Delta u(\cdot,s),\zeta(\cdot,s)\rangle ds+\int_0^t\langle u(\cdot,s), \partial_s\zeta(\cdot,s)\rangle_{-1}^1ds \notag\\
&+\int_0^tf(s)\langle g(\cdot),\zeta(\cdot,s)\rangle dW(s),
\end{align}
where $H^{-1}(D)$ denotes the dual space of $H_0^1(D)$ and $\langle\cdot,\cdot\rangle_{-1}^1$ denotes the dual product between spaces $H^1_0(D)$ and $H^{-1}(D)$.

\item[(ii)] If $m=2$, for any $\zeta\in L^2([0,T];H^1_0(D))$ with $\partial_t \zeta \in L^2([0,T];L^2(D))$ and $\partial_{tt} \zeta \in L^2([0,T];H^{-1}(D))$, it holds  for almost every $t\in[0,T]$ that
\begin{align*}
\langle \partial_tu(\cdot,t), \zeta (\cdot,t)\rangle- \langle u(\cdot,t), \partial_t\zeta (\cdot,t)\rangle = \int_0^t\langle\Delta u(\cdot,s),\zeta(\cdot,s)\rangle ds
\\
-\int_0^t\langle u(\cdot,s), \partial_{ss} \zeta(\cdot,s)\rangle_{-1}^1 ds
+\int_0^tf(s)\langle g(\cdot),\zeta(\cdot,s)\rangle dW(s).
\end{align*}

\end{theorem}

\begin{proof}
We only present the proof for the parabolic case (i) below, and the proof of the hyperbolic case (ii) follows similarly. 

We start with the special case $\zeta(x,t)=\eta(x)\phi(t)$, where $\eta\in H^1_0(D)$ and $\phi\in C^1([0,T])$. Denote by $G(x,t):=\phi(t)x$ and 
\[
Y(t):=\langle u(\cdot,t), \eta(\cdot)\rangle = \int_0^t \langle \Delta u(\cdot,s), \eta(\cdot)\rangle ds +\langle g,\eta\rangle \int_0^tf(s)dW(s)
\]
for any $t\in[0,T]$. 
Utilizing It\^{o}'s formula yields 
\begin{eqnarray*}
&&\langle u(\cdot,t),\zeta(\cdot,t)\rangle = G(Y(t),t) \\
&=& \int_0^t\left[\frac{\partial G}{\partial s}(Y(s),s)+\frac{\partial G}{\partial x}(Y(s),s)\langle \Delta u(\cdot,s),\eta(\cdot)\rangle\right] ds + \langle g,\eta \rangle\int_0^t\frac{\partial G}{\partial x}(Y(s),s) f(s)dW(s) \\
&=& \int_0^t\left[ \phi'(s)Y(s)+\phi(s)\langle \Delta u(\cdot,s),\eta(\cdot)\rangle \right] ds + \int_0^t\phi(s)\langle g,\eta \rangle f(s)dW(s)\\
&= &\int_0^t\langle u(\cdot,s), \partial_s \zeta(\cdot,s)\rangle ds+ \int_0^t\langle\Delta u(\cdot,s),\zeta(\cdot,s)\rangle ds+ \int_0^tf(s)\langle g(\cdot),\zeta(\cdot,s)\rangle dW(s),
\end{eqnarray*}
which indicates that \eqref{prop:para_strong} holds with $\langle u(\cdot,s), \partial_s \zeta(\cdot,s)\rangle_{-1}^1=\langle u(\cdot,s), \partial_s \zeta(\cdot,s)\rangle$ for $\zeta(x,s)=\eta(x)\phi(s)$ and $\phi\in C^1([0,T])$. 
Denote the space
\[
\mathcal A:=\text{Span} \{\eta\cdot\phi: \eta\in H^1_0(D), \phi\in C^1([0,T])\}.
\]
Then for any $\zeta\in \mathcal A$, \eqref{prop:para_strong} is also satisfied due to the linearity of the integrals. 

Next, for a general function
\[
\zeta \in W([0,T]):=\{\zeta : \zeta \in L^2([0,T];H_0^1(D)) , \partial_t \zeta \in L^2([0,T];H^{-1}(D))\},
\]
equipped with the norm
\[
\|\zeta\|_{W([0,T])} := \Big(\|\zeta\|^2_{L^2([0,T];H^1(D))}+ \|\partial_t \zeta\|^2_{L^2([0,T];H^{-1}(D))}\Big)^{1/2},
\]
we claim that the space \(\mathcal{A}\) is dense in \(W([0,T])\). That is, there exists a sequence \(\{\zeta_n\}_{n\in\mathbb N} \subset \mathcal{A}\) such that
\begin{align}\label{eq:claim}
\lim_{n\to\infty}\|\zeta_n-\zeta\|_{L^2([0,T];H^1(D))} = 0, \quad
\lim_{n\to\infty}\|\partial_t(\zeta_n-\zeta)\|_{L^2([0,T];H^{-1}(D))} = 0 .
\end{align}
In fact, by \cite[Lemma 2.11]{FK01}, the space \(\mathcal{A}\) is dense in \(C^1([0,T];H^1_0(D))\). On the other hand, it follows from \cite[Theorem 2.1]{LM72} that \(C_0^\infty([0,T];H^1_0(D))\) is dense in \(W([0,T])\), which implies that \(C^1([0,T];H^1_0(D))\) is dense in \(W([0,T])\). Since the embedding
\[
C^1([0,T];H^1_0(D)) \hookrightarrow W([0,T])
\]
is continuous, the claim follows.

Based on the claim above, one can find a proper subsequence $\{\zeta_{n_k}\}_{k\in \mathbb{N}}\subset\{\zeta_n\}_{n\in\mathbb{N}}$ such that 
\begin{align}\label{eq:claim2}
\lim_{k\to\infty}\|\zeta_{n_k}(\cdot,t)-\zeta(\cdot,t)\|_{H^1(D)}=0,\quad \lim_{k\to\infty}\|\partial_t(\zeta_n(\cdot,t)-\zeta(\cdot,t))\|_{H^{-1}(D)}=0
\end{align}
for almost every $t\in[0,T]$. 
Then for almost every $t\in[0,T]$, it holds
\begin{align*}
\langle u(\cdot,t), \zeta (\cdot,t)\rangle &=\lim_{k\rightarrow \infty} \langle u(\cdot,t), \zeta_{n_k} (\cdot,t)\rangle \\
&= \lim_{k\rightarrow \infty}\bigg[\int_0^t\langle\Delta u(\cdot,s),\zeta_{n_k}(\cdot,s)\rangle ds+\int_0^t\langle u(\cdot,s), \partial_s \zeta_{n_k}(\cdot,s)\rangle_{-1}^1 dt\\
&\quad +\int_0^tf(s)\langle g(\cdot),\zeta_{n_k}(\cdot,s)\rangle dW(s)\bigg]\\
&= \int_0^t\langle\Delta u(\cdot,s),\zeta(\cdot,s)\rangle ds+\int_0^t\langle u(\cdot,s), \partial_s \zeta(\cdot,s)\rangle_{-1}^1 ds
+\int_0^tf(s)\langle g(\cdot),\zeta(\cdot,s)\rangle dW(s),
\end{align*}
which verifies \eqref{prop:para_strong} and completes the proof. 
Here, the second equality holds due to the fact that \eqref{prop:para_strong} holds for each $\zeta_{n_k}\in\mathcal A$. The last equality holds, i.e.,
\begin{align*}
&\lim_{k\to\infty}\int_0^t\langle \Delta u(\cdot,s), \zeta_{n_k}(\cdot,s)\rangle ds =\int_0^t\langle \Delta u(\cdot,s), \zeta(\cdot,s)\rangle ds,  \\
&\lim_{k\to\infty}\int_0^t\langle u(\cdot,s), \partial_s \zeta_{n_k}(\cdot,s)\rangle_{-1}^1 ds =\int_0^t\langle u(\cdot,s), \partial_s \zeta(\cdot,s)\rangle_{-1}^1 ds,\\
&\lim_{k\to\infty}\int_0^t\langle g(\cdot), \zeta_{n_k}(\cdot,s)\rangle f(s) dW(s) =\int_0^t\langle g(\cdot), \zeta(\cdot,s)\rangle f(s) dW(s),
\end{align*}
because of the following estimates:
\begin{align*}
\mathbb E\left[\left|\int_0^t\langle \Delta u(\cdot,s), \zeta(\cdot,s)-\zeta_{n_k}(\cdot,s)\rangle ds\right|^2\right] 
&\le\mathbb E\left[\left( \int_0^t\|\Delta u(\cdot,s)\|_{L^2(D)}\|\zeta(\cdot,s)-\zeta_{n_k}(\cdot,s)\|_{L^2(D)}ds\right)^2\right]\\
&\le\|u\|_{L^2(\Omega;L^2([0,T];\dot H^2(D)))}^2\|\zeta-\zeta_{n_k}\|_{L^2([0,T];L^2(D))}^2 \rightarrow 0,
\end{align*}
\begin{align*}
\mathbb E\left[\left|\int_0^t\langle  u(\cdot,s), \partial_s(\zeta-\zeta_{n_k})(\cdot,s)\rangle_{-1}^1 ds\right|^2\right] 
&\le\mathbb E\left[\left(\int_0^t\| u(\cdot,s)\|_{ H^1(D)}\|\partial_s(\zeta-\zeta_{n_k})(\cdot,s)\|_{H^{-1}(D)}ds\right)^2\right]\\
&\le\|u\|_{L^2(\Omega;L^2([0,T];\dot H^1(D)))}^2\|\partial_t(\zeta-\zeta_{n_k})\|_{L^2([0,T];H^{-1}(D))}^2 \rightarrow 0,
\end{align*}
\begin{align*}
\mathbb{E}\left[\left|\int_0^t\langle g(\cdot), \zeta(\cdot,s)-\zeta_{n_k}(\cdot,s)\rangle f(s) dW(s)\right|^2\right]
&= \int_0^t\langle g(\cdot), \zeta(\cdot,s)-\zeta_{n_k}(\cdot,s)\rangle^2 f^2(s) ds\\
&\le \|g\|_{L^2(D)}^2\int_0^t  \|\zeta(\cdot,s)-\zeta_{n_k}(\cdot,s)\|_{L^2(D)}^2 f^2(s) ds \rightarrow 0
\end{align*}
as $k\to\infty$, where \eqref{eq:claim} and \eqref{eq:claim2} are utilized.
\end{proof}

For $m=1$ and any $\zeta \in L^2([0,T]; H^1_0(D))$ with $\partial_t \zeta \in L^2([0,T]; H^{-1}(D))$, it follows from \cite[Section 5.9.2, Theorem 3]{E98} that there exists a modification $\tilde{\zeta}$ of $\zeta$ such that $\tilde{\zeta} \in C([0,T]; L^2(D))$. For $m=2$ and any $\zeta \in L^2([0,T]; H^1_0(D))$ with $\partial_t \zeta \in L^2([0,T]; L^{2}(D))$ and $\partial_{tt} \zeta \in L^2([0,T]; H^{-1}(D))$, it follows from \cite[Section 5.9.2, Theorem 2]{E98} that there exists a modification $\tilde{\zeta}$ of $\zeta$ such that $\tilde{\zeta} \in C([0,T]; L^2(D))$ and $\partial_t\tilde\zeta\in C([0,T];H^{-1}(D))$.
In the subsequent analysis of the inverse problem, we will use a time-convolution form of Theorem~\ref{thm:strong} applied to the smooth modification $\tilde\zeta$, which is stated in the following corollary.

\begin{corollary}\label{cor:strong}
Let $\{u(\cdot,t)\}_{t\in[0,T]}$ be the stochastic strong solution of \eqref{eq:model} with $f\in L^2([0,T])$ and $g\in \dot{H}^1(D)$. 
\item[(i)] If $m=1$, for any $\zeta$ satisfies conditions in Theorem \ref{thm:strong} (i) with $\zeta \in C([0,T]; L^2(D))$, it holds for any $t\in[0,T]$ that
\begin{align*}
\langle u(\cdot,t), \zeta(\cdot,0)\rangle
&= \int_0^t \langle \Delta u(\cdot,s), \zeta(\cdot,t-s)\rangle ds-\int_0^t \langle u(\cdot,s), \partial_t \zeta(\cdot,t-s)\rangle_{-1}^1 ds \\
&\quad + \int_0^t f(s)\langle g(\cdot), \zeta(\cdot,t-s)\rangle dW(s).
\end{align*}

\item[(ii)]  If $m=2$, for any $\zeta$ satisfies conditions in Theorem \ref{thm:strong} (ii) with $\zeta \in C([0,T]; L^2(D))$ and $\partial_t\zeta \in C([0,T]; H^{-1}(D))$, it holds for any $t\in[0,T]$ that

\begin{align*}
\langle \partial_t u(\cdot,t), \zeta(\cdot,0)\rangle 
 - \langle u(\cdot,t), \partial_t \zeta(\cdot,0)\rangle
&= \int_0^t \langle \Delta u(\cdot,s), \zeta(\cdot,t-s)\rangle ds -\int_0^t \langle u(\cdot,s), \partial_{tt} \zeta(\cdot,t-s)\rangle_{-1}^1ds \\
&\quad + \int_0^t f(s)\langle g(\cdot), \zeta(\cdot,t-s)\rangle dW(s).
\end{align*}
\end{corollary}

\begin{proof}
We only prove the parabolic case (i) with $m=1$. For any fixed $t\in[0,T]$, define 
$\eta(\cdot,s) := \zeta(\cdot,t-s)$.  
Then $\eta$ satisfies the conditions in Theorem \ref{thm:strong} (i) and in addition $\eta\in C([0,T];L^2(D))$. 
Hence, \eqref{prop:para_strong} can be extended to any $t\in[0,T]$, since the convergence of  the subsequence $\{\zeta_{n_k}\}_{k\in \mathbb{N}}\subset\{\zeta_n\}_{n\in\mathbb{N}}$ in \eqref{eq:claim2} is valid for any $t\in[0,T]$ under the addition assumption that $\zeta\in C([0,T];L^2(D))$. 
As a result, for any $t\in[0,T]$, it holds
\begin{align*}
\langle u(\cdot,t), \eta(\cdot,t)\rangle
&=\int_0^t \langle \Delta u(\cdot,s), \eta(\cdot,s)\rangle ds+\int_0^t \langle u(\cdot,s), \partial_s \eta(\cdot,s)\rangle_{-1}^1ds\\
&\quad + \int_0^t f(s)\langle g(\cdot), \eta(\cdot,s)\rangle dW(s),
\end{align*}
which concludes the proof due to the facts $\eta(\cdot,t)=\zeta(\cdot,0)$ and $\partial_s \eta(\cdot,s)=-\partial_t\zeta(\cdot,t-s)$.
\end{proof}

\section{Uniqueness for the inverse problem}
This section is devoted to the uniqueness for recovering the strength of $f$ involved in the random source. We first introduce several notations and some preliminary results. Then by establishing the relationship between the 
boundary flux $\left.\frac{\partial u}{\partial \overrightarrow{\mathbf{n}}}\right|_{\partial D}$ and the unknown strength $|f|$, we give the proof of Theorem \ref{thm:inverse}, which shows the uniqueness of the inverse problem for determining the strength $|f|$ in \eqref{eq:model}.

\subsection{The basis $\{\varphi_l\}_{l\in\mathbb N^+}$ of $L^2(\partial D)$}
For the eigensystem $\left\{\lambda_n, \phi_n\right\}_{n\in\mathbb N^+}$ of the operator $-\Delta$ with Dirichlet boundary condition, the trace theorem yields that $\left\{\left.\frac{\partial \phi_n}{\partial \overrightarrow{\mathbf{n}}}\right|_{\partial D}\right\}_{n\in\mathbb N^+} \subset H^{\frac12}(\partial D)$. For $\phi,\psi \in L^2(\partial D)$, their inner product is defined as 
$$
\langle \phi,\psi \rangle_{L^2(\partial D)}:=\int_{\partial D} \phi(x)\psi(x)dS(x)
$$
with the induced norm $\|\cdot\|^2_{L^2(\partial D)}=\langle\cdot,\cdot\rangle_{L^2(\partial D)}$. The following lemmas concern the non-vanishing property of $\left\{\left.\frac{\partial \phi_n}{\partial \overrightarrow{\mathbf{n}}}\right|_{\partial D}\right\}_{n\in\mathbb N^+}$ and the density of its linear span.

\begin{lemma}(\cite[Lemma 2.1]{LZZ22})\label{lemma_vanish}
If $\Gamma$ is a nonempty open subset of $\partial D$, then for each $n \in \mathbb{N}^{+}, \frac{\partial \phi_n}{\partial \overrightarrow{\mathbf{n}}}$ can not vanish almost everywhere on $\Gamma$.
\end{lemma}

\begin{lemma}(\cite[Lemma 2.2]{LZZ22})
The set $\operatorname{Span}\left\{\left.\frac{\partial \phi_n}{\partial \overrightarrow{\mathbf{n}}}\right|_{\partial D}:n\in\mathbb N^+\right\}$ is dense in $L^2( \partial D)$.
\end{lemma}

From the above lemmas, we are able to construct the orthonormal basis $\{\varphi_l\}_{l\in\mathbb N^+}$ in $L^2( \partial D)$. Firstly, we set $\varphi_1=\left.\frac{\partial \phi_1}{\partial \overrightarrow{\mathbf{n}}}\right|_{\partial D} /\left\|\frac{\partial \phi_1}{\partial \overrightarrow{\mathbf{n}}}\right\|_{L^2( \partial D)}$, and assume that the orthonormal set $\{\varphi_j\}_{j=1}^{l-1}$ has been built for some $l\ge2$. To construct $\varphi_l$, we choose the smallest number, denoted by $n_l \in \mathbb{N}^{+}$, such that 
\[
\left.\frac{\partial \phi_{n_l}}{\partial \overrightarrow{\mathbf{n}}}\right|_{\partial D} \notin \operatorname{Span}\{\varphi_1,\ldots,\varphi_{l-1}\}, 
\]
and pick $\varphi_l \in \operatorname{Span}\left\{\left.\frac{\partial \phi_{n_l}}{\partial \overrightarrow{\mathbf{n}}}\right|_{\partial D}, \varphi_1, \ldots, \varphi_{l-1}\right\}$ satisfying $\|\varphi_l\|_{L^2(\partial D)}=1$ and
$$
\langle\varphi_l, \varphi_j\rangle_{L^2( \partial D)}=0,\quad\forall~ j=1, \ldots, l-1.
$$
The density of the set $\operatorname{Span}\left\{\left.\frac{\partial \phi_n}{\n}\right|_{\partial D}:n\in\mathbb N^+\right\}$ in $L^2( \partial D)$ yields that $\{\varphi_l\}_{l\in\mathbb N^+}$ is an orthonormal basis in $L^2( \partial D)$. Also, functions in the basis $\{\varphi_l\}_{l\in\mathbb N^+}$ has the same regularity as those in $\left\{\left.\frac{\partial \phi_n}{\partial \overrightarrow{\mathbf{n}}}\right|_{\partial D}\right\}_{n\in\mathbb N^+}$, i.e., $\{\varphi_l\}_{l\in\mathbb N^+}\subset H^{\frac12}(\partial D)$.

\subsection{The coefficients $\{c_{z,n}\}_{n\in\mathbb N^+}$}
For each $l \in \mathbb{N}^{+}$, denote by $\xi_l \in H^1(D)$ the weak solution of the system:
\begin{equation*}
\left\{\begin{aligned}
-\Delta \xi_l(x) & =0, & & x \in D, \\
\xi_l(x) & =\varphi_l(x), & & x \in \partial D .
\end{aligned}\right.
\end{equation*}
For $z\in\partial D$, we define 
\begin{equation}\label{psi_zN}
\psi_z^N(x):=\sum_{l=1}^N \varphi_l(z) \xi_l(x),\ x\in D
\end{equation}
and the auxiliary coefficients
\[
c_{z,n}^N:=\left\langle\psi_z^N, \phi_n\right\rangle,\quad n\in\mathbb N^+.
\]

The following lemma shows that the limit of $c_{z,n}^N$ exists as $N\rightarrow \infty$, with the limit being related to the normal derivative $\frac{\partial \phi_n}{\partial \overrightarrow{\mathbf{n}}}$ at $z$.

\begin{lemma}\label{Lemma:cz}
The limit of the coefficients $c_{z,n}^N$ exists as $N\to\infty$ in the sense of $L^2(\partial D)$. Moreover, denoting the limit by $c_{z,n}:=\lim_{N\to\infty}c_{z,n}^N$, it holds
\[
c_{z,n}=-\frac{1}{\lambda_n}\frac{\partial \phi_n}{\partial \overrightarrow{\mathbf{n}}}(z).
\]
\end{lemma}

\begin{proof}
From Green's theorem, we have	
\begin{align}\label{eq:cznN}
c_{z,n}^N&=\left\langle\psi_z^N, \phi_n\right\rangle= \frac{1}{\lambda_n}\langle \psi_z^N, -\Delta \phi_n\rangle\notag\\
&= \frac{1}{\lambda_n}\left(\langle-\Delta\psi_z^N, \phi_n\rangle+\left\langle \frac{\partial \psi_z^N}{\partial \overrightarrow{\mathbf{n}}}, \phi_n\right\rangle_{L^2(\partial D)} - \left\langle \psi_z^N, \frac{\partial \phi_n}{\partial \overrightarrow{\mathbf{n}}}\right\rangle_{L^2(\partial D)}\right)\notag\\
&= -\frac{1}{\lambda_n}\int_{\partial D}\sum_{l=1}^{N}\varphi_l(z)\xi_l(x)\frac{\partial \phi_n}{\partial \overrightarrow{\mathbf{n}}}(x) dS(x)\notag\\
&= -\frac{1}{\lambda_n}\sum_{l=1}^{N}\varphi_l(z)\left\langle\varphi_l, \frac{\partial \phi_n}{\partial \overrightarrow{\mathbf{n}}}\right\rangle_{L^2(\partial D)},
\end{align}
where the fact $\Delta\psi_z^N(x)=\sum_{l=1}^N\varphi_l(z)\Delta\xi_l(x)=0$ for $x\in D$ and the boundary condition of $\phi_n$ are used in the third step, and the boundary condition $\xi_l|_{\partial D}=\varphi_l|_{\partial D}$ is used in the last step.  Since $\{\varphi_l\}_{l\in\mathbb N^+}$ is an orthonormal basis in $L^2( \partial D)$ and $\frac{\partial \phi_n}{\partial \overrightarrow{\mathbf{n}}}\in H^{\frac12}(\partial D)\subset L^2(\partial D)$, the proof is thus completed by taking $N\to\infty$. 
\end{proof}

\subsection{Uniqueness for recovering $|f|$}
Before showing the uniqueness of the inverse problem for determining the strength $|f|$ involved in the random source in \eqref{eq:model}, we establish the relationship between the measurement, i.e., the boundary flux data $\left.\frac{\partial u}{\partial \overrightarrow{\mathbf{n}}}\right|_{\partial D}$, and the unknown strength $|f|$. 

Consider the following auxiliary system with $m\in\{1,2\}$:
\begin{equation}\label{model_uzN}
\left\{\begin{aligned}
	\left(\partial_t^m-\Delta\right) u_z^N(x, t) & =0, & & (x, t) \in D \times(0, T), \\
	u_z^N(x, t) & =0, & & (x, t) \in \partial D \times(0, T), \\
	u_z^N(x, 0) & =-\psi_z^N(x), & & x \in D ,\\
	\partial_t u_z^N(x, 0) &= 0, & & x \in D,\quad \text{(only if $m=2$)}.
\end{aligned}\right.
\end{equation}
Since $\psi_z^N=\sum_{l=1}^N \varphi_l(z)\xi_l \in H^1(D)$ with $\xi_l \in H^1(D)$, there exists a weak solution $u_z^N$ with the following regularity:
\[
\begin{cases}
u_z^N \in C([0,T];L^2(D)), & m=1,\\[6pt]
u_z^N \in C([0,T];L^2(D)),\ \partial_t u_z^N \in C([0,T];H^{-1}(D)) & m=2,
\end{cases}
\]
see \cite[Section 5.9.2, Theorem 3 and Section 7.1.2, Theorem 3]{E98} for the case $m=1$, and \cite[Theorem 2.3]{LLT86} for the case $m=2$. Hence, Corollary~\ref{cor:strong} can be applied to weak solution $u_z^N$.

The following representation lemma gives a stochastic integral representation of the boundary flux $\left.\frac{\partial u}{\partial \overrightarrow{\mathbf{n}}}\right|_{\partial D}$ averaged over time.

\begin{lemma}\label{lem:representaition}
Let $\{u(\cdot,t)\}_{t\in[0,T]}$ be the stochastic strong solution of \eqref{eq:model} with $f\in L^2([0,T])$ and $g\in \dot{H}^1(D)$, and define 
\[
\omega_z^N(x,t)=\psi_z^N(x)+u_z^N(x,t), 
\]
where $u_z^N$ and $\psi_z^N$ are given in \eqref{psi_zN} and \eqref{model_uzN}, respectively. For almost every $t\in [0,T]$, it holds
$$
-\int_{0}^{t} \frac{\partial u}{\partial \overrightarrow{\mathbf{n}}} (z,\tau) d\tau = \lim_{N\rightarrow \infty} \int_{0}^{t}f(\tau)\langle g,\omega_z^N(\cdot,t-\tau)\rangle dW(\tau),
$$
where the limit on the right hand side is taken in the sense of $L^2(\partial D)$ with respect to $z\in\partial D$.
\end{lemma}

\begin{proof}
According to \eqref{psi_zN} and \eqref{model_uzN}, together with the fact $\Delta\psi_z^N=\sum_{l=1}^N\varphi_l(z)\Delta\xi_l(x)=0$ for $x\in D$, it is easy to check that $w_z^N$ is a weak solution of the equation
\[
\left\{\begin{aligned}
\left(\partial_t^m-\Delta\right) \omega_z^N(x, t) & =0, & & (x, t) \in D \times(0, T), \\
\omega_z^N(x, t) & =\psi_z^N(x), & & (x, t) \in \partial D \times(0, T), \\
\omega_z^N(x, 0) & =0, & & x \in D ,\\
\partial_t\omega_z^N(x, 0) &= 0, & & x \in D,\quad \text{(only if $m=2$)},
\end{aligned}\right.
\]
where the existence of $\omega_z^N$ can be established by arguments analogous to those used for $u_z^N$ above.

From the definition of the stochastic strong solution ${u(\cdot,t)}_{t\in[0,T]}$, we have for $t\in[0,T]$ and $m\in\{1,2\}$ that
\begin{equation}\label{eq:4-1}
\int_0^tf(\tau)\langle g, \psi_z^N\rangle dW(\tau)= \langle \partial_t^{m-1} u(\cdot, t),\psi_z^N(\cdot)\rangle-\int_0^t\langle \Delta u(\cdot, \tau), \psi_z^N(\cdot)\rangle d\tau.
\end{equation}
For $u_z^N \in C([0,T];L^2(D))$ with in addition $\partial_tu_z^N\in C([0,T]; H^{-1}(D))$ if $m=2$, Corollary \ref{cor:strong} yields for $t\in[0,T]$ and $m\in\{1,2\}$ that 
\begin{align}\label{eq:4-2}
\int_0^tf(\tau)\langle g(\cdot), u_z^N(\cdot, t-\tau)\rangle dW(\tau)=& -\langle \partial_t^{m-1} u(\cdot,t), \psi_z^N(\cdot)\rangle - \int_0^t\langle \Delta u(\cdot, \tau), u_z^N(\cdot, t-\tau)\rangle\notag\\
&+\int_0^t \langle u(\cdot,s), \partial_t^m  u_z^N(\cdot, t-\tau)\rangle_{-1}^1 d\tau.
\end{align}
Combining \eqref{eq:4-1} with \eqref{eq:4-2}, and recalling the definition of $\omega_z^N$, we obtain
\begin{align*}
\int_0^tf(\tau)\langle g(\cdot), \omega_z^N(\cdot, t-\tau)\rangle dW(\tau)&= -\int_0^t\langle \Delta u(\cdot, \tau), \omega_z^N(\cdot,t-\tau)\rangle d\tau +\int_0^t \langle u(\cdot,\tau), \partial_t^m  \omega_z^N(\cdot, t-\tau)\rangle_{-1}^1 d\tau\\
&= -\int_0^t\langle \Delta u(\cdot, \tau), \omega_z^N(\cdot,t-\tau)\rangle d\tau -\int_0^t \langle \nabla u(\cdot,\tau), \nabla  \omega_z^N(\cdot, t-\tau)\rangle d\tau \\
&= - \int_0^t \left\langle \frac{\partial u}{\partial \overrightarrow{\mathbf{n}}}(\cdot, \tau), \psi_z^N \right\rangle_{L^2(\partial D)} d \tau
\end{align*}
for $t\in[0,T]$, where Green's Theorem and the boundary condition of $\omega_z^N$ are utilized in the last equality.

From the boundary condition of $\psi_z^N$ and $\xi_l$, we obtain that
$$
\psi_z^N(x)=\sum_{l=1}^N \varphi_l(z) \varphi_l(x), \quad x \in \partial D.
$$
Then it holds 
\begin{align}\label{rep3}
 \int_0^t  f(\tau) \langle g(\cdot),\omega_z^N(\cdot, t-\tau) \rangle d W(\tau) 
 =&- \int_0^t \left\langle \frac{\partial u}{\partial \overrightarrow{\mathbf{n}}}(\cdot, \tau), \psi_z^N \right\rangle_{L^2(\partial D)} d \tau\notag\\
=&-\int_0^t \sum_{l=1}^N \varphi_l(z)\left\langle\frac{\partial u}{\partial \overrightarrow{\mathbf{n}}}(\cdot, \tau), \varphi_l\right\rangle_{L^2(\partial D)} d \tau.
\end{align}

Denote
\[
S_N(z):=\int_0^t \sum_{l=1}^N \varphi_l(z)\left\langle\frac{\partial u}{\partial \overrightarrow{\mathbf{n}}}(\cdot, \tau), \varphi_l\right\rangle_{L^2(\partial D)} d \tau.
\]
We claim that $\{S_N(z)\}_{N=1}^\infty$ forms a Cauchy sequence in $L^2(\partial D)$. If the claim holds, by taking the limit in $L^2(\partial D)$ on both sides of \eqref{rep3} as $N\to\infty$, we obtain the desired result. In fact, by the trace theorem, the regularity property 
$u\in L^2(\Omega; L^2([0,T];\dot{H}^2(D)))$, and the elliptic regularity estimate 
$\|u(\cdot, \tau)\|_{H^2(D)}\leq C\|\Delta u(\cdot, \tau)\|_{L^2(D)}$ (cf. \cite[Section 6.2.3, Theorem 6 and Section  6.3.2, Theorem 4]{E98}), we deduce that for $t\in[0,T]$,
\[
\Big\|\int_0^t\frac{\partial u}{\partial \overrightarrow{\mathbf{n}}}(\cdot, \tau)d\tau \Big\|_{L^2(\partial D)}^2
\leq C\int_0^t\|u(\cdot, \tau)\|_{H^2(D)}^2d\tau
\leq C\int_0^t\|u(\cdot, \tau)\|_{\dot{H}^2(D)}^2d\tau <\infty.
\]
Hence, for $N_1,N_2\in \mathbb{N}$ with $N_1<N_2$, we obtain
\begin{align*}
\left\| \sum_{l=N_1}^{N_2}\int_0^t \varphi_l(\cdot)\left\langle \frac{\partial u}{\partial \overrightarrow{\mathbf{n}}}(\cdot, \tau), \varphi_l\right\rangle_{L^2(\partial D)}d\tau \right\|_{L^2(\partial D)}^2 
&= \left\| \sum_{l=N_1}^{N_2} \varphi_l(\cdot)\left\langle \int_0^t \frac{\partial u}{\partial \overrightarrow{\mathbf{n}}}(\cdot, \tau)d\tau, \varphi_l\right\rangle_{L^2(\partial D)} \right\|_{L^2(\partial D)}^2 \\
&= \sum_{l=N_1}^{N_2} \left\langle \int_0^t \frac{\partial u}{\partial \overrightarrow{\mathbf{n}}}(\cdot, \tau)d\tau, \varphi_l\right\rangle_{L^2(\partial D)}^2 \longrightarrow 0
\end{align*}
as $N_1, N_2 \to \infty$, which verifies the claim. The proof is thereby completed.
\end{proof}

From the above lemma, the following corollary can be deduced.

\begin{corollary}
Under the settings in Lemma \ref{lem:representaition}, assume further that $g \in \dot{H}^{d-1+\epsilon}(D)$ for some arbitrary $\epsilon > 0$. Then for almost every $t \in[0, T]$, it holds
\begin{equation}\label{rep4}
-\int_{0}^{t} \frac{\partial u}{\partial \overrightarrow{\mathbf{n}}} (z,\tau) d\tau =\int_0^t f(\tau) G_z(t-\tau) d W(\tau),
\end{equation}
where 
\begin{equation}\label{def:Gz}
G_z(t) =\left\{\begin{aligned}
&\sum_{n=1}^{\infty} c_{z, n} g_n\left(1-e^{-\lambda_n t}\right), \quad &m=1,\\
&\sum_{n=1}^{\infty} c_{z, n} g_n\left(1-\cos(\sqrt{\lambda_n}t)\right), \quad &m=2.
\end{aligned}\right.
\end{equation}
\end{corollary}

\begin{proof}
Based on Parseval's equality, Green's theorem, the definition of $\omega_z^N$, and Lemma \ref{Lemma:cz}, we derive
\begin{align*}
\left\langle g(\cdot), \omega_z^N(\cdot,t-\tau)\right\rangle&= \sum_{n=1}^{\infty}g_n\left\langle \phi_n, \omega_z^N(\cdot,t-\tau)\right\rangle= \sum_{n=1}^{\infty} -\frac{g_n}{\lambda_n}\left\langle \Delta\phi_n, \omega_z^N(\cdot,t-\tau)\right\rangle\\
&= \sum_{n=1}^{\infty}\frac{g_n}{\lambda_n}\left[\-\left\langle \nabla\phi_n, \nabla\omega_z^N(\cdot,t-\tau)\right\rangle - \left\langle \frac{\partial \phi_n}{\partial \overrightarrow{\mathbf{n}}}, \omega_z^N(\cdot,t-\tau)\right\rangle_{L^2(\partial D)}\right]\\
&= \sum_{n=1}^{\infty}g_n\left[-\frac{1}{\lambda_n}\partial_t^m \left\langle \phi_n, \omega_z^N(\cdot,t-\tau)\right\rangle +c_{z,n}^N\right],
\end{align*}
where we used \eqref{eq:cznN} and the boundary condition of $\omega_z^N$ in the last step.
It then gives that
$$
\left\langle g(\cdot), \omega_z^N(\cdot, t-\tau)\right\rangle_{L^2(D)} =\left\{\begin{aligned}
&\sum_{n=1}^{\infty} c_{z, n}^N g_n\left(1-e^{-\lambda_n(t-\tau)}\right), \quad &m=1,\\
&\sum_{n=1}^{\infty} c_{z, n}^N g_n\left(1-\cos(\sqrt{\lambda_n}(t-\tau))\right), \quad &m=2.
\end{aligned}\right.
$$

Next, we will show $\|\sum_{n=1}^{\infty}g_nc_{\cdot,n} \|_{L^2(\partial D)}<\infty$ such that the dominated convergence theorem can be applied to the above series. 
Since $g\in \dot{H}^{d-1+\epsilon}$, there exists $d-1<r<d-1+\epsilon$ such that $\|g\|_{\dot{H}^r(D)}<\infty$. Then
$$
g_n=\langle g,\phi_n \rangle= \lambda_n^{-r/2} \langle g,(-\Delta)^{r/2}\phi_n \rangle = \lambda_n^{-r/2} \langle (-\Delta)^{r/2} g,\phi_n \rangle \leq \lambda_n^{-r/2}\| g\|_{\dot{H}^r(D)}.
$$
Using Lemma \ref{Lemma:cz} and the estimate of normal derivative $\|\frac{\partial \phi_n}{\partial \overrightarrow{\mathbf{n}}}\|_{L^2(\partial D)} \leq C\lambda_n^{\frac12}$ (cf.\cite[Theorem 1.1]{HT02}) with the constant $C>0$ independent of $n$, we obtain
$$
\|c_{\cdot,n}\|_{L^2(\partial D)} = \frac{1}{\lambda_n}\left\|\frac{\partial \phi_n}{\partial \overrightarrow{\mathbf{n}}}\right\|_{L^2(\partial D)}\leq C\lambda_n^{-\frac12}.
$$
Hence by Weyl's law ($\lambda_n\sim n^{\frac{2}{d}}$ as $n\rightarrow \infty$), it holds  
\[
\left\|\sum_{n=1}^{\infty}g_nc_{\cdot,n} \right\|_{L^2(\partial D)}\leq C\sum_{n=1}^{\infty} \lambda_n^{-\frac{1+r}{2}}\leq C \sum_{n=1}^{\infty} n^{-\frac{1+r}{d}} < \infty.
\]
The proof is then completed.
\end{proof}

Utilizing the representation of boundary flux data \eqref{rep4}, we are now in the position to give the proof of uniqueness theorem for inverse problem.

\begin{proof}[Proof of Theorem \ref{thm:inverse}]
From \eqref{rep4} and It{\^{o}}'s isometry, it follows that
\begin{equation}\label{recover:f}
\mathbb{V}\left[ \int_0^t \frac{\partial u}{\partial \overrightarrow{\mathbf{n}}}(z,\tau) \, d\tau \right] = \int_0^t f^2(\tau) G_z^2(t - \tau) \, d\tau,\quad z\in D_{ob}.
\end{equation}
Lemmas \ref{lemma_vanish} and \ref{Lemma:cz} ensure that $G_z$ defined in \eqref{def:Gz} does not vanish on $D_{ob}$. Then the uniqueness of $|f|$ is given by the Titchmarsh convolution theorem (cf.\cite[Lemma 4.1]{GHBKS14}), which completes the proof.
\end{proof}

\section{Numerical reconstructions.}
In this section, we present numerical reconstructions of \( |f(t)| \) from boundary flux measurements, utilizing the representation formula \eqref{recover:f} in both one-dimensional and two-dimensional cases. To avoid the inverse crime, we discretize the original equation \eqref{eq:model} using the finite difference method, rather than discretizing the explicit representation of the solution as given in the definition of the mild solution \eqref{solution:mild}. The numerical solution for each sample path is computed using the implicit Euler scheme, and the boundary normal derivative data are then approximated by the numerical solution. For the inverse problem, formula \eqref{recover:f} can be viewed as a Volterra equation of the first kind, where the recovery of $|f(t)|$ is known to be ill-posed. To deal with this ill-posedness, we employ the Kaczmarz iterative algorithm and verify its stability under random perturbations through numerical experiments.

\subsection{Synthetic data}
We begin by describing the numerical discretization of the direct problem. For simplicity, the procedure is illustrated in the one-dimensional setting with $D=[0,1]$.

Let \( N_t \) and \( N_x \) denote the amounts of temporal and spatial partitions, respectively. Define the temporal and spatial step sizes as \( h_t \) and \( h_x \), and introduce the discretization nodes:
\[
t_j = j h_t, \quad j = 0, 1, \dotsc, N_t; \qquad x_i = i h_x, \quad i = 0, 1, \dotsc, N_x.
\]
Let \( u_i^j \) denote the numerical approximation of \( u(x_i, t_j) \). The second-order spatial and temporal derivatives are discretized using central finite difference scheme:
\[
\delta_x^2 u_i^j := \frac{1}{h_x^2} \left( u_{i+1}^j - 2u_i^j + u_{i-1}^j \right), \quad \delta_t^2 u_i^j := \frac{1}{h_t^2} \left( u_{i}^{j+1} - 2u_i^j + u_{i}^{j-1} \right).
\]

Applying the implicit Euler scheme, we obtain the following numerical schemes for the direct problem \eqref{eq:model}. For the parabolic case with \( m = 1 \), the scheme reads:
\[
\left\{
\begin{aligned}
u_i^{j+1} - u_i^j - h_t \delta_x^2 u_i^{j+1} &= f(t_j) g(x_i) \Delta W(t_j), && i = 1,\ldots, N_x - 1,\quad j = 0,\ldots, N_t - 1, \\
u_i^0 &= 0,\quad u_0^j = u_{N_x}^j = 0, && i = 0,\ldots, N_x,\quad j = 1,\ldots, N_t,
\end{aligned}
\right.
\]
while for the hyperbolic case with \( m = 2 \), the scheme becomes:
\[
\left\{
\begin{aligned}
\delta_t^2 u_i^j - \frac{1}{4}\left(\delta_x^2 u_i^{j+1} + 2\delta_x^2 u_i^j + \delta_x^2 u_i^{j-1}\right) &= f(t_j) g(x_i) \frac{\Delta W(t_j)}{h_t}, && i = 1,\ldots, N_x - 1,\quad j = 1,\ldots, N_t - 1, \\
\frac{2}{h_t^2} u_i^1 - \frac{1}{2} \delta_x^2 u_i^1 &= f(t_0) g(x_i) \frac{\Delta W(t_0)}{h_t}, && i = 1,\ldots, N_x - 1, \\
u_i^0 &= 0,\quad u_0^j = u_{N_x}^j = 0, && i = 0,\ldots, N_x,\quad j = 1,\ldots, N_t,
\end{aligned}
\right.
\]
where \( \Delta W(t_j) = W(t_{j+1}) - W(t_j) \) denotes the increment of the Brownian motion. The second equation in the hyperbolic case is to simulate the numerical solution at \( t = t_1 \), and is derived by applying the first equation at \( j = 0 \), combining with
 \( u_i^0 = 0 \) and $u_i^1=u_i^{-1}$ due to the initial conditions $u(x,0)=0$ and \( \partial_t u(x,0) = 0 \) for $x\in D$.

To generate the boundary data required for the inverse problem, we approximate the boundary flux \( \frac{\partial u}{\partial \overrightarrow{\mathbf{n}}} \) using the following finite differences:
\[
\frac{\partial u}{\partial \overrightarrow{\mathbf{n}}} (0,t_j) \approx -\frac{u_1^j - u_0^j}{h_x}, \quad \frac{\partial u}{\partial \overrightarrow{\mathbf{n}}} (1,t_j) \approx \frac{u_{N_x - 1}^j - u_{N_x}^j}{h_x}
\]
and define the noisy data
\begin{align}\label{eq:noisydata}
\frac{\partial u^\sigma}{\partial \overrightarrow{\mathbf{n}}} (0,t_j):=-\frac{u_1^j - u_0^j}{h_x}+\sigma U,\quad \frac{\partial u^\sigma}{\partial \overrightarrow{\mathbf{n}}} (1,t_j):=\frac{u_{N_x - 1}^j - u_{N_x}^j}{h_x}+\sigma U
\end{align}
for $j = 1, \ldots, N_t$, where $\sigma>0$ denotes the noise level and $U$ denotes a random variable satisfying the uniform distribution on $[-1,1]$.

These approximations enable us to simulate the boundary data numerically appearing on the left-hand side of equation \eqref{recover:f}.

\subsection{Discretization of the recovery formula}
To discretize the recovery formula \eqref{recover:f}, we fix \( z \in \partial D= \{0,1\} \). 
For \( t = t_j \) with $j=1,\ldots,N_t$, the left-hand side of \eqref{recover:f} is approximated by
\[
\mathbb{V} \left[ \int_0^{t_j} \frac{\partial u}{\partial \overrightarrow{\mathbf{n}}}(z,\tau)d\tau \right]\approx \frac{h_t^2}P\sum_{p=1}^P \left[ \sum_{k=0}^{j-1} \frac{\partial u}{\partial \overrightarrow{\mathbf{n}}}(z, t_k,\omega_p) \right]^2,
\]
where $P$ denotes the number of sample paths and $ \frac{\partial u}{\partial \overrightarrow{\mathbf{n}}}(z,0,\omega_p)$ denotes the observation of the boundary flux on the sample path $\omega_p\in\Omega$, and the right-hand side can be approximated by
\[
\int_0^{t_j} f^2(\tau) G_z^2(t_1 - \tau)d\tau \approx  h_t \sum_{k=0}^{j-1} f^2(t_k) G_z^2(t_{j-k}).
\]
Combining the above two formulas, we propose the numerical recovery formula based on the noisy data defined in \eqref{eq:noisydata}:
\[
V_j:=\frac{h_t}P\sum_{p=1}^P \left[ \sum_{k=0}^{j-1} \frac{\partial u^\sigma}{\partial \overrightarrow{\mathbf{n}}}(z, t_k,\omega_p) \right]^2= \sum_{k=0}^{j-1} f_k^2 G_z^2(t_{j-k}),\quad j=1,\ldots, N_t,
\]
where $f_k^2$ is an approximation of $f^2(t_k)$. 
The numerical recovery formula can also be rewritten into the following discrete linear system at each observation point \( z \in \partial D \):
\begin{equation}\label{eq:V} 
\mathbf{V}_z = \mathbf{G}_z \mathbf{f},
\end{equation}
where $\mathbf{V}_z=[V_1,\ldots,V_{N_t}]^\top$, \( \mathbf{f}= [f_0^2, f_1^2, \dotsc, f_{N_t-1}^2]^\top \), and
\( \mathbf{G}_z=\left[G_{jk}\right]_{1\le j,k\le N_t} \in \mathbb{R}^{N_t \times N_t} \) is a lower triangular matrix with 
\begin{equation*}
G_{jk}=\left\{
\begin{aligned}
&G_z^2(t_{j-k+1}),\quad& j\ge k,\\
&0,\quad& j<k.
\end{aligned}
\right.
\end{equation*}
To numerically recover $|f(t)|$, it then suffices to solve \eqref{eq:V} for ${\bf f}$.

Note that the matrix  \( \mathbf{G}_z \) is highly ill-conditioned, and hence the inverse problem is ill-posed. To deal with the ill-posedness of the inverse problem, we adopt the Kaczmarz iterative method, utilizing measurements at different observation points \( z \in \partial D = \{0,1\} \).
Specifically, starting from an initial guess \( \mathbf{f}^{(0)} = \mathbf{0} \), we employ a regularized block Kaczmarz iterative scheme that alternately uses the measurements at \( z = 0 \) and \( z = 1 \). The update at the \( k \)th iteration is given by
\[
\mathbf{f}^{(k+1)} = \mathbf{f}^{(k)} + \left( \mathbf{G}_{z_k}^\top \mathbf{G}_{z_k} + \alpha \mathbf{I} \right)^{-1} \mathbf{G}_{z_k}^\top \left( \mathbf{V}_{z_k} - \mathbf{G}_{z_k} \mathbf{f}^{(k)} \right),
\]
where \( z_k = 0 \) for even \( k \), and \( z_k = 1 \) for odd \( k \). The regularization parameter \( \alpha > 0 \) improves the stability of the iteration. The stopping criterion is based on the combined residual:
\[
\|\mathbf{G}_0 \mathbf{f}^{(k)} - \mathbf{V}_0\| + \|\mathbf{G}_1 \mathbf{f}^{(k)} - \mathbf{V}_1\| < \varepsilon
\]
for some given tolerance parameter $\varepsilon>0$.

\subsection{Numerical examples}

In this part, we present several numerical examples to demonstrate the validity and effectiveness of the proposed method. The first example illustrates the inversion results in the one-dimensional parabolic setting under varying number of sample paths and noise levels. The second example demonstrates the performance of the method in the one-dimensional hyperbolic case, while the last example addresses the two-dimensional parabolic case. 

{\bf Example 1} (One-dimensional parabolic case): 
We choose the spatial coefficient function as \( g(x) = x(1 - x) \) for $x\in D=(0,1)$. To generate the synthetic data, we set $h_x=2^{-6}$ and $h_t=2^{-7}$. The regularization parameter is set as \( \alpha = 10^{-2} \) and the tolerance parameter is chosen as $\varepsilon=2\times10^{-3}$. 

\begin{figure}[H]
\begin{tabular}{ccc}
    \includegraphics[ width=0.3\textwidth]{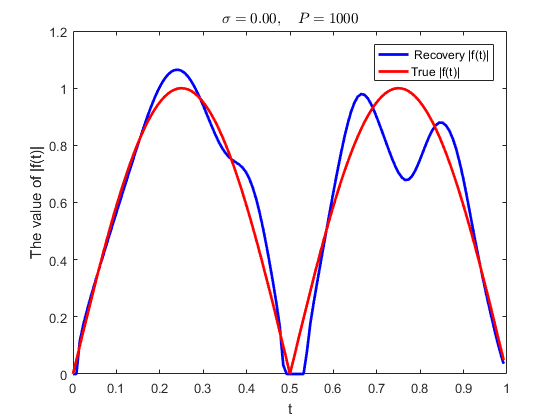}
    \includegraphics[ width=0.3\textwidth]{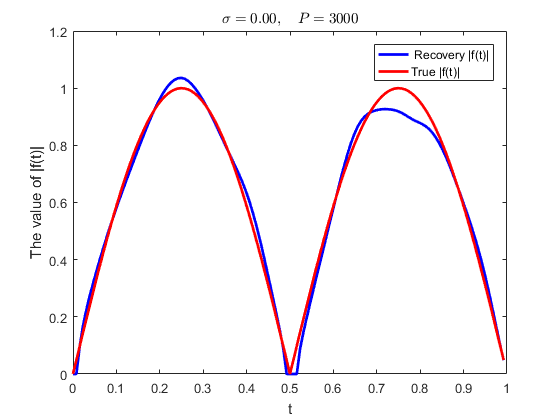}
    \includegraphics[ width=0.3\textwidth]{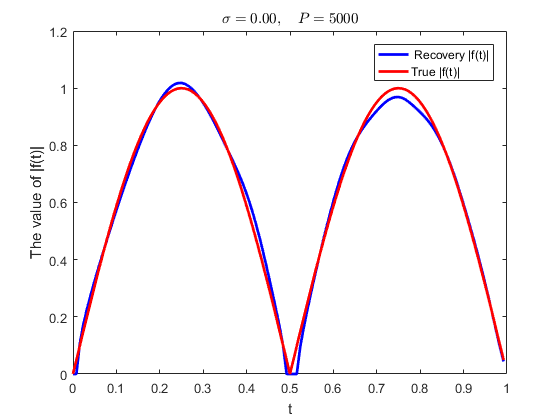} 
\end{tabular}
\caption{Reconstruction results of $|f| $ with $P=1000,3000,5000$ ($\sigma=0$).}
\label{fig:ex1_1}
\end{figure}

In Figures \ref{fig:ex1_1} and \ref{fig:ex1_2}, we consider the reconstruction of $|f(t)|$ with
\[
f(t) =  \sin(2\pi t),\quad t\in[0,1].
\]
Figure \ref{fig:ex1_1} shows the recovery of $|f|$ under different number of sample paths $P=1000,3000,5000$ with the noise level $\sigma=0$. Figure \ref{fig:ex1_2} displays the reconstructed function \( |f| \) under  noise levels $\sigma=0.05,0.1,0.2$ with $P=5000$ sample paths. 
As observed from the figures, the reconstruction becomes more accurate as we impose more realizations, and remains stable under increasing noise.

\begin{figure}[h]
\begin{tabular}{ccc}
    \includegraphics[ width=0.3\textwidth]{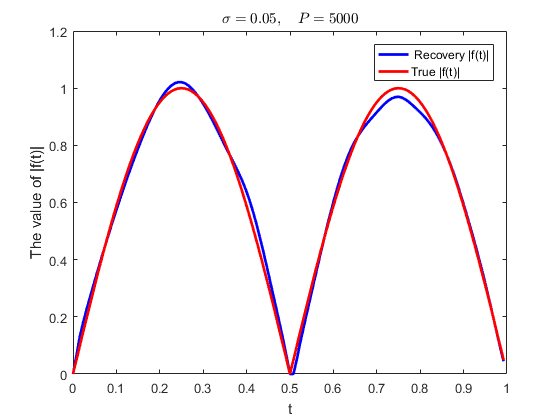}
    \includegraphics[ width=0.3\textwidth]{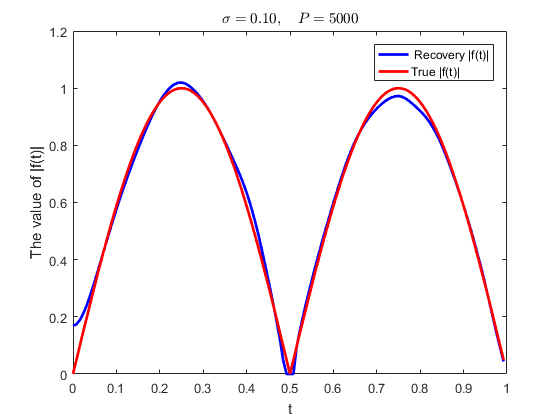}
    \includegraphics[ width=0.3\textwidth]{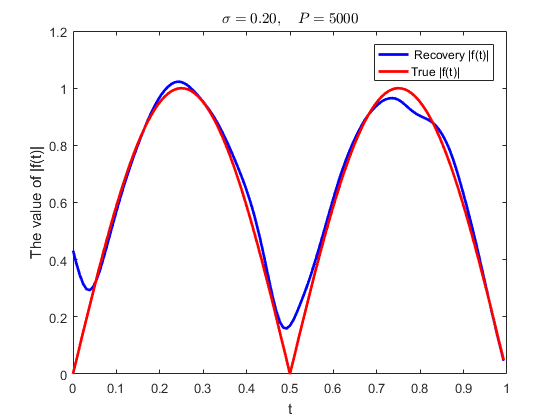} 
\end{tabular}
\caption{Reconstruction results of $|f| $ with $\sigma=0.05,0.1,0.2$ ($P=5000$).}
\label{fig:ex1_2}
\end{figure}

In Figure \ref{fig:ex2}, we consider a more oscillatory $f$ containing two Fourier modes:
\[
f(t) = 0.6 - 0.3 \cos(2\pi t) - 0.3 \cos(4\pi t),\quad t\in[0,1],
\]  
and a discontinuous $f$:
\[
f(t) =  1.5\chi_{\{0.2 < t \leq 0.6\}} + 1 \chi_{\{0.6 < t \leq 0.8\}},\quad t\in[0,1].
\] 
We choose the number of sample paths as $P=10000$ and the noise level \( \sigma = 0 \). The numerical results demonstrate that our inversion algorithm is capable of recovering more oscillatory \( f \).

\begin{figure}[H]
\begin{tabular}{ccc}
    \includegraphics[ width=0.35\textwidth]{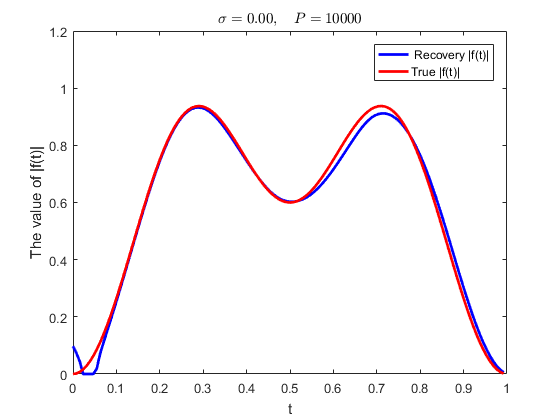}
    \includegraphics[ width=0.35\textwidth]{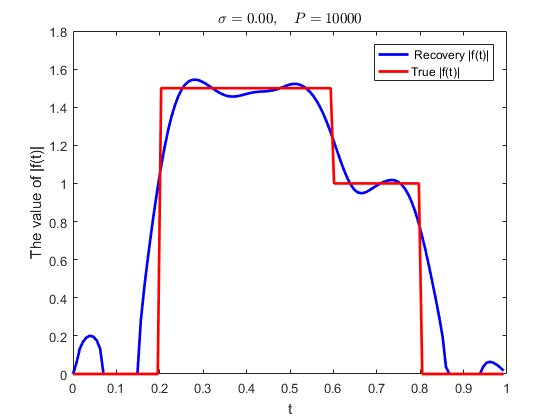} 
\end{tabular}
\caption{Reconstruction results of $|f|$ with $\sigma=0$ and $P=10000$.}
\label{fig:ex2}
\end{figure}

{\bf Example 2} (One-dimensional hyperbolic case): Reconstruction of \( |f| \) with
\[
f(t) = 0.6 - 0.3 \cos(2\pi t) - 0.3 \cos(4\pi t),\quad t\in[0,1],
\]  
and
\[
f(t) =  1.5\chi_{\{0.2 < t \leq 0.6\}} + 1 \chi_{\{0.6 < t \leq 0.8\}},\quad t\in[0,1].
\] 
Set \( g(x) = x(1 - x)\) for \(x\in D = (0,1) \),  \(h_x = 2^{-6},  h_t = 2^{-7},\alpha = 5\times10^{-7}, \varepsilon=10^{-4} \), $\sigma=0$, and $P=10000$. The reconstruction result of \( |f| \) is shown in Figure \ref{fig:ex4}, which indicates the accuracy of proposed method for hyperbolic case.

\begin{figure}[H]
\begin{tabular}{ccc}
    \includegraphics[ width=0.35\textwidth]{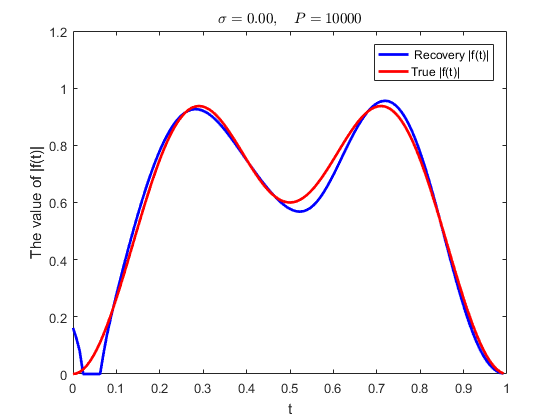}
    \includegraphics[ width=0.35\textwidth]{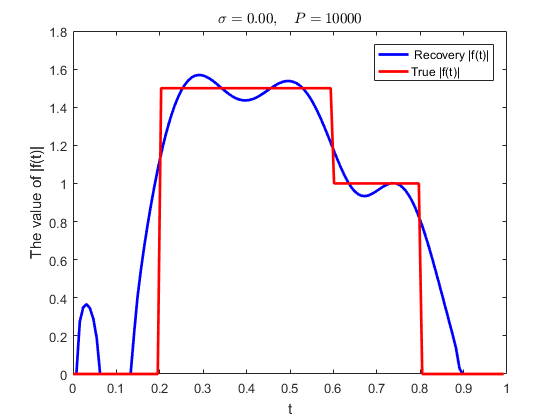} 
\end{tabular}
\caption{Reconstruction results of $|f|$ with $\sigma=0$ and $P=10000$.}
\label{fig:ex4}
\end{figure}

{\bf Example 3} (Two-dimensional parabolic case): Reconstruction of \( |f| \) with
\[
f(t) = 0.6 - 0.3 \cos(2\pi t) - 0.3 \cos(4\pi t),\quad t\in[0,1].
\]  
We set \( g(x,y) = xy(1 - x)(1-y) \) for  \((x,y)\in D = (0,1)\times (0,1) \), $h_x=h_y=2^{-5}$, $h_t=2^{-7}$, \( \alpha = 10^{-9} \), \(\varepsilon=10^{-6}\), \(\sigma=0\), and $P=5000$. The reconstruction result of \(|f|\) is shown in Figure~\ref{fig:ex3} by using different observation points. The left figure is generated by measurements at a single observation point \(z=(0,0.2)\), while the right one shows the result with observation points \(z_1=(0,0.2),\ z_2=(0.6,0)\) and \(z_3=(1,0.4)\). 
The reconstruction achieved with multiple observation points demonstrates enhanced accuracy, highlighting the significant role of spatially enriched data in improving inversion performance.

\begin{figure}[H]
\begin{tabular}{ccc}
    \includegraphics[ width=0.35\textwidth]{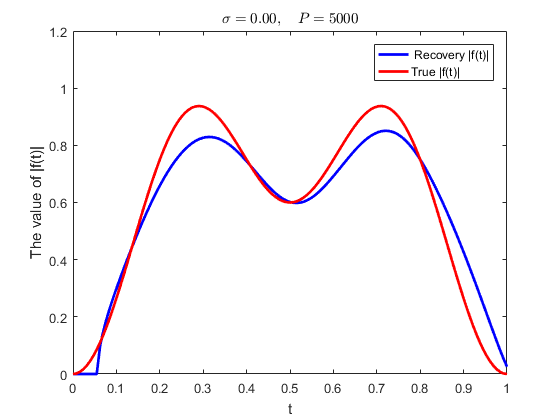}
    \includegraphics[ width=0.35\textwidth]{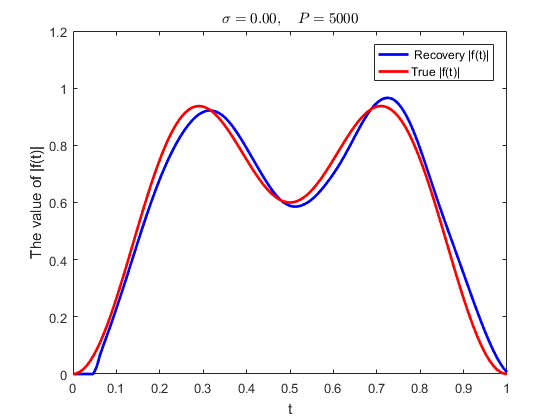} 
\end{tabular}
\caption{Reconstruction results of $|f|$ with $\sigma=0$ and $P=5000$.}
\label{fig:ex3}
\end{figure}

\section{Conclusion and future works.}

In this paper, we have studied the inverse random source problem associated with stochastic heat and wave equations driven by a random source of the form \( g(x)f(t)\dot{W}(t) \). 
A novel framework has been introduced to uniquely determine the strength $|f|$ of the unknown time-dependent function $f$ in the random source. It demonstrates the well-posedness of the stochastic strong solutions, establishes a stochastic integral-by-parts-type formula, and develops a recovery formula. The recovery formula reveals the relationship between the boundary flux and the stochastic convolution via an auxiliary function argument, and indicates that $|f|$ can be uniquely determined by the boundary flux $\frac{\partial u}{\partial \overrightarrow{\mathbf{n}}}$ over the measurement domain $D_{ob}\times(0,T)$.

For inverse random source problems related to time-dependent equations, several challenges remain unsolved:
\begin{itemize}
\item[(1)] \textbf{Unique recovery of both $f(t)$ and $g(x)$.} While this work focuses on the recovery of the time-dependent component, it remains an open question whether the spatial profile \( g(x) \) can also be uniquely and stably identified. This problem has been studied in deterministic settings (cf. \cite{RZ20, LZZ22}) under proper regularity assumptions on the source, which makes the framework inapplicable to the random case due to the roughness of the random noise.

\item[(2)] \textbf{Promotion to more complicated stochastic systems.} For example, it is of interest to investigate whether the uniqueness result can be extended to nonlinear equations or equations driven by more general noise structures such as space-time nonseparable noise or fractional noise.
\end{itemize}
We hope to be able to report the progress on these problems elsewhere in the future.


\end{document}